\documentclass[12pt]{amsart}
\usepackage{amsthm, amssymb, amsmath,dsfont}
\usepackage{mdwlist}
\usepackage{mathrsfs}
\usepackage{a4wide}
\usepackage[T1]{fontenc}

\newtheorem{theorem}{Theorem}
\newtheorem*{thm}{H\"older's inequality for $\boldsymbol{0<p<1}$}

\theoremstyle{definition}

\theoremstyle{remark}

\newcommand{\Z}{\mathbb{Z}}
\newcommand{\Q}{\mathbb{Q}}
\newcommand{\R}{\mathbb{R}}

\newcommand{\da}{\mathrm{d}}

\newcommand{\p}{\varphi}
\newcommand{\Pa}{\mathbb{P}}
\newcommand{\pp}{\boldsymbol{p}}
\newcommand{\w}{\widetilde}
\newcommand{\Sa}{\mathcal{S}}
\newcommand{\ind}{\mathds{1}}

\renewcommand{\geq}{\geqslant}
\renewcommand{\leq}{\leqslant}

\begin{document}
\title[Characterisation of $L_p$-norms]{Characterisation of $L_p$-norms\\ via H\"older's inequality}

\author[T. Kochanek]{Tomasz Kochanek}
%\address{Institute of Mathematics, University of Silesia, Bankowa 14, 40-007 Katowice, Poland}
%\email{t\_kochanek@wp.pl}
\author[M. Lewicki]{Micha\l\ Lewicki}
\address{Institute of Mathematics, University of Silesia, Bankowa 14, 40-007 Katowice, Poland}
\email{t\_kochanek@wp.pl, m\_lewicki@wp.pl}

\date{}
\dedicatory{}

\begin{abstract}
We characterise $L_p$-norms on the space of integrable step functions, defined on a~probabilistic space, via H\"older's type inequality with an~optimality condition.
\end{abstract}

\maketitle
%%%%%%%%%%%%%%%%%%%%%%%%%%%%%%%%%%%%%%%%%%%
%%%%%%%%%%%%%%%%%%%%%%%%%%%%%%%%%%%%%%%%%%%
\section{Introduction}
In a series of papers Matkowski (\cite{matkowski (pams)}, \cite{matkowski (aeq)}, \cite{matkowski (studia)}, \cite{matkowski (indag)}, \cite{jmaa}), jointly with \'Swi\k{a}tkowski (\cite{matkowski_swiatkowski (jmaa)}, \cite{matkowski_swiatkowski}), derived several characterisations of the $L_p$-norm via classical H\"older's and Minkowski's inequalities. In this paper we will deal with a~certain, in a~sense critical, case concerning the H\"older inequality.

Hereinafter $(X,\Sigma,\mu)$ stands for a~measure space. Every $\mu$-integrable function will be treated as an element of $L_1(\mu)$, {\it i.e.}, we interpret equality between such functions in the $\mu$-almost everywhere sense. We denote $\Sa=\Sa(X)$ the vector space of all $\Sigma$-integrable step functions and $\Sa_+=\Sa_+(X)=\{f\in\Sa\colon f\geq 0\}$. Denote $\R_+=[0,\infty)$ and for any bijection $\p\colon\R_+\to\R_+$ with $\p(0)=0$ set $$\Pa_\p(f)=\p^{-1}\Bigl(\int_X\p\circ f\,\da\mu\Bigr)\quad\mbox{for }f\in\Sa.$$

We are motivated by the following result of Matkowski:
\begin{theorem}[{\cite[Theorem 3]{matkowski (studia)}}]\label{matkowski1}
Suppose that there are two sets $A,B\in\Sigma$ such that
\begin{equation}\label{AB}
0<\mu(A)<1<\mu(B)<\infty
\end{equation}
and $\p,\psi\colon\R_+\to\R_+$ are bijections satisfying $\p(0)=\psi(0)=0$ and 
\begin{equation}\label{A}
\int_X fg\,\da\mu\leq\Pa_\p(f)\Pa_\psi(g)\quad\mbox{for }f,g\in\Sa_+.
\end{equation}
Then there exist numbers $p,q>1$ with $p^{-1}+q^{-1}=1$ such that $\p(t)=\p(1)t^p$ and $\psi(t)=\psi(1)t^q$ for $t\in\R_+$.
\end{theorem}
\noindent
It was also shown in \cite{matkowski (studia)} that assumption \eqref{AB} in the above theorem is essential. Namely, we have what follows:
\begin{theorem}[{\cite[Theorem 5]{matkowski (studia)}}]\label{matkowski2}
Suppose that $(X,\Sigma,\mu)$ is a~probabilistic space {\rm (}i.e. $\mu(X)=1${\rm )} such that for at least one set $A\in\Sigma$ we have $0<\mu(A)<1$. Then, bijections $\p,\psi\colon\R_+\to\R_+$ with $\p(0)=\psi(0)=0$ satisfy inequality \eqref{A} if and only if the map $F\colon\R_+^2\to\R_+$ defined by $F(s,t)=\p^{-1}(s)\psi^{-1}(t)$ is concave.
\end{theorem}
As we see, the condition $\mu(X)=1$ is critical and in this case H\"older's inequality does not determine a~concrete form of $\p$ and $\psi$. The aim of this paper is to show that this situation may be fixed by adopting additionally the following consistency condition:
$$\begin{array}{l}
\mbox{For every non-zero function }f\in\Sa_+\mbox{ there exists a~function }\chi\colon\R_+\to\R_+\\
\mbox{satisfying }\chi(0)=0\mbox{ and }\chi\circ f\not=0\mbox{ and such that inequality \eqref{A} becomes}\\
\mbox{equality for }g=\chi\circ f.\mbox{ Conversely, for every non-zero }g\in\Sa_+\mbox{ there is a}\\
\mbox{function }\tau\colon\R_+\to\R_+\mbox{ satisfying }\tau(0)=0\mbox{ and }\tau\circ g\not=0\mbox{ and such that}\\
\mbox{inequality \eqref{A} becomes equality for }f=\tau\circ g.
\end{array}
\leqno(*)
$$
It simply says that inequality \eqref{A} is assumed to be optimal for any given map from $\Sa$. Recall that the equality case in H\"older's inequality for exponents $p$ and $q$ occurs exactly when $f^p$ and $g^q$ are proportional, hence the above condition holds true with $\chi(t)=t^{p/q}$ and $\tau(t)=t^{q/p}$. That is why, we believe, it is a natural requirement.

Our idea is based on introducing a~new measure space $(Y,T,\nu)$, with total mass greater than $1$, and then applying a~result (Theorem \ref{M1}) concerning generalised Minkowski's inequality for the product space $(X\times Y,\Sigma\otimes T,\mu\otimes\nu)$. Let us recall that the mentioned inequality asserts that if $\mu$ and $\nu$ are $\sigma$-finite measures and $F\colon X\times Y\to\R_+$ is $(\Sigma\otimes T)$-measurable, then for every $1\leq p<\infty$ we have
\begin{equation}\label{GMI}
\Biggl\{\int_X\Biggl(\int_YF(x,y)\,\nu(\da y)\Biggr)^p\mu(\da x)\Biggr\}^{1/p}\leq\int_Y\Biggl(\int_XF(x,y)^p\,\mu(\da x)\Biggr)^{1/p}\nu(\da y).
\end{equation}
%%%%%%%%%%%%%%%%%%%%%%%%%%%%%%%%%%%%%%%%%%%
%%%%%%%%%%%%%%%%%%%%%%%%%%%%%%%%%%%%%%%%%%%
\section{Results}
\begin{theorem}\label{M1}
Let $(X,\Sigma,\mu)$ and $(Y,T,\nu)$ be measure spaces such that $[0,1]\subset\mu(\Sigma)$ and $[0,\alpha]\subset\nu(T)$ for some $\alpha>1$. Suppose $\p,\psi\colon\R_+\to\R_+$ satisfy $\p(0)=\psi(0)=0$ and $\p(\R_+)=[0,\beta|$ {\rm (}may be either right-closed or right-open{\rm )} for some $\beta\in (0,\infty]$. Then, the inequality
\begin{equation}\label{pp}
\psi\Biggl\{\int_X\p\Biggl(\int_YF(x,y)\,\nu(\da y)\Biggr)\mu(\da x)\Biggr\}\leq\int_Y\psi\Biggl(\int_X\p\circ F(x,y)\,\mu(\da x)\Biggr)\nu(\da y)
\end{equation}
holds true for every $F\in\Sa_+(X\times Y)$ if, and only if, either:
\begin{itemize*}
\item[{\rm (i)}] $\psi|_{\p(\R_+)}=0$, or
\item[{\rm (ii)}] $\p(1)\not=0\not=\psi(1)$ and there exists $p\geq 1$ such that $\p(t)=\p(1)t^p$ and $\psi(t)=\psi(1)t^{1/p}$ for $t\in\R_+$.
\end{itemize*} 
\end{theorem}
\begin{proof}
For arbitrary $A_1\in\Sigma$ and $B_1\in T$, and any $t\in\R_+$, we have
\begin{equation*}
\psi\bigl(\mu(A_1)\p(\nu(B_1)t)\bigr)\leq\nu(B_1)\psi\bigl(\mu(A_1)\p(t)\bigr),
\end{equation*}
which follows from \eqref{pp} by taking $F=t\ind_{A_1\times B_1}$. Hence, $$\psi(a\p(bt))\leq b\psi(a\p(t))\quad\mbox{for }a\in [0,1], b\in [0,\alpha]\mbox{ and }t\in\R_+.$$Therefore, if $b,t>0$ then we have $$\frac{\psi(a\p(bt))}{bt}\leq\frac{\psi(a\p(t))}{t}\, ,$$whence substituting $v=bt$ (then $v/t=b\in [0,\alpha]$) gives
\begin{equation}\label{psiv}
\frac{\psi(a\p(v))}{v}\leq\frac{\psi(a\p(t))}{t}\quad\mbox{for }a\in [0,1]\mbox{ and }t,v>0\mbox{ such that }\frac{v}{t}\leq\alpha.
\end{equation}

For any $\delta>0$ and $a\in [0,1]$ define a~map $\Phi_{\delta,a}\colon [\delta,\alpha\delta]\to\R_+$ by $$\Phi_{\delta,a}(x)=\frac{\psi(a\p(x))}{x}\, .$$Since $x/y\leq\alpha$ for every $x$ and $y$ from $[\delta,\alpha\delta]$, inequality \eqref{psiv} implies that $\Phi_{\delta,a}$ is constant. Hence, there is $M(\delta,a)\in\R_+$ satisfying
\begin{equation}\label{psiM}
\psi(a\p(x))=M(\delta,a)x\quad\mbox{for }a\in [0,1],\,\delta>0\mbox{ and }x\in [\delta,\alpha\delta].
\end{equation}
Taking any sequence $(\delta_n)_{n=-\infty}^{\infty}$ of positive numbers with $$\delta_{-n}\xrightarrow[n\to\infty]{}0,\quad\,\delta_n\xrightarrow[n\to\infty]{}0\,\,\quad\mbox{and}\,\,\,\delta_n<\delta_{n+1}<\alpha\delta_n\,\mbox{ for }n\in\Z,$$we can see that the numbers $M(\delta,a)$ in equation \eqref{psiM} do not depend on $\delta$. Thus, there is a~function $M\colon [0,1]\to\R_+$ such that
\begin{equation}\label{pM}
\psi(a\p(x))=M(a)x\quad\mbox{for }a\in [0,1]\mbox{ and }x\in\R_+.
\end{equation}
If $M(a)=0$ for some $a\in (0,1]$, then \eqref{pM} would imply that $\psi$ vanishes on the interval $\p(\R_+)$, that is, assertion (i) holds true. So, for the rest of the proof we may assume that $M(a)\not=0$ for every $a\in (0,1]$.

Let $a,b\in [0,1]$. Composing the two functions: $x\mapsto\psi(a\p(x))$ and $x\mapsto\psi(b\p(x))$, and using \eqref{pM}, we obtain $$\psi(a\p(\psi(b\p(x))))=M(a)M(b)x\quad\mbox{for }x\in\R_+.$$Fixing for a~moment the variable $x$ and regarding the both sides of the above equation as functions of $a$ and $b$, we conclude by symmetry that
\begin{equation}\label{pab}
\psi(a\p(\psi(b\p(x))))=\psi(b\p(\psi(a\p(x)))).
\end{equation}
Now, observe that $\psi$ is a~one-to-one function on the set $\{a\p(x)\colon a\in [0,1], x\in\R_+\}$. For if $0\leq a\leq b\leq 1$ and $x,y\in\R_+$ satisfy $a\p(x)\not=b\p(y)$, then $a/b\leq 1$, so $a/b\cdot\p(x)\in\p(\R_+)$, say $\p(z)=a/b\cdot\p(x)$. Of course, $b\not=0$, so $M(b)\not=0$. Thus, $$\psi(a\p(x))=\psi(b\p(z))=M(b)z\not=M(b)y=\psi(b\p(y)).$$Consequently, equation \eqref{pab} implies that $a\p(\psi(b\p(x)))=b\p(\psi(a\p(x)))$, that is, $$\frac{\p\circ\psi(a\p(x))}{a}=\frac{\p\circ\psi(b\p(x))}{b}\quad\mbox{for }a,b\in (0,1]\mbox{ and }x\in\R_+.$$Hence, there is a~map $\gamma\colon\R_+\to\R_+$ satisfying $$\p\circ\psi(a\p(x))=a\gamma(x)\quad\mbox{for }a\in [0,1]\mbox{ and }x\in\R_+.$$

Now, for any $a,b\in [0,1]$ and $x\in\R_+$ we have $$M(a)M(b)x=\psi(b\p(\psi(a\p(x))))=\psi(ab\gamma(x)),$$thus there is a~map $N\colon [0,1]\to\R_+$ such that $\psi(c\gamma(x))=N(c)x$ for every $c\in [0,1]$, $x\in\R_+$ and, moreover, $$N(ab)=M(a)M(b)\quad\mbox{and}\quad  N(a)=M(1)M(a)\quad\mbox{for }a,b\in [0,1].$$Therefore, the function $m\colon (0,1]\to\R_+$ given by $m(x)=M(x)/M(1)$ is multiplicative ({\it i.e.} satisfies $m(xy)=m(x)m(y)$) and does not vanish, which implies that $$m(x)=(\exp\circ A\circ\log)(x)\quad\mbox{for }x\in (0,1],$$
where $A\colon\R\to\R$ is an additive function (\cite[Theorem 13.1.3]{kuczma}).

Using equation \eqref{pM} we get
\begin{equation}\label{paMx}
\psi(a\p(x))=M(a)x=M(1)m(a)x=M(1)(\exp\circ A\circ\log)(a)x
\end{equation}
for $a\in (0,1]$ and $x\in\R_+$. Putting here $x=1$ yields $$\psi(\p(1)a)=M(1)(\exp\circ A\circ\log)(a)\quad\mbox{for }a\in (0,1]$$(which, in particular, implies $\p(1)\not=0$), thus 
\begin{equation}\label{Zpsi}
\psi(z)=M(1)(\exp\circ A\circ\log)\Bigl(\frac{z}{\p(1)}\Bigr)\quad\mbox{for }z\in (0,\p(1)].
\end{equation}

Fix any $x>0$ and pick $a\in (0,1]$ such that $a\p(x)\leq\p(1)$. Applying \eqref{Zpsi} to $z=a\p(x)$ we obtain $$\psi(a\p(x))=M(1)(\exp\circ A\circ\log)\Bigl(a\frac{\p(x)}{\p(1)}\Bigr).$$On the other hand, we have formula \eqref{paMx}, and consequently, $$A\Bigl(\log a\frac{\p(x)}{\p(1)}\Bigr)=A(\log a)+\log x,$$that is $$A(\log\p(x))=A(\log\p(1))+\log x\quad\mbox{for }x\in (0,\infty).$$The last equation implies that $A$ is surjective and also injective on the interval $(-\infty,\log\beta|$ (as $x\in (0,\infty)$ the values $\p(x)$ runs through $(0,\beta|$). This implies that $A$ is injective on the whole real line (otherwise for some $h_1,\ldots ,h_k$ from a~Hamel basis of $\R$ over $\Q$, and some $\lambda_1,\ldots ,\lambda_k\in\Q$, not all equal to zero, we would have $\sum_j\lambda_jf(h_j)=0$, but then $f(k\sum_j\lambda_jh_j)=0$ for every $k\in\Q$, so $f$ would vanish at infinitely many points of $(-\infty,\log\beta|$). Applying the function $\exp\circ A^{-1}$ to the both sides of the equation above we get
$$
\p(x)=\p(1)\exp\bigl\{A^{-1}(\log x)\bigr\}\quad\mbox{for }x\in (0,\infty)
$$
and, coming back to \eqref{paMx}, we get
$$
\psi(x)=\frac{M(1)}{\p(1)}\exp\bigl\{A(\log x)\bigr\}\quad\mbox{for }x\in (0,\infty).
$$
Replacing, with no loss of generality, $\p$ by $\p/\p(1)$ and $\psi$ by $\psi/\psi(1)$ we may conclude that both $\p$ and $\psi$ are multiplicative on $(0,\infty)$ and $\psi=\p^{-1}$.

Now, pick any sets $A_1,A_2\in\Sigma$ with $A_1\cap A_2=\emptyset$ and $B_1,B_2\in T$ with $B_1\cap B_2=\emptyset$ such that $$b:=\mu(A_1)=\mu(A_2)>0\quad\mbox{and}\quad c:=\nu(B_1)=\nu(B_2)>0$$(we may take $b=c=1/2$). For arbitrary $t,u,v,w\in\R_+$ define a~map $F\in\Sa_+$ by
$$
F(x,y)=\left\{\begin{array}{cl}
G(x) & \mbox{if }y\in B_1,\\
H(x) & \mbox{if }y\in B_2,\\
0 & \mbox{if }y\in Y\setminus(B_1\cup B_2),
\end{array}\right.
$$
where $$G=t\ind_{A_1}+u\ind_{A_2}\quad\mbox{and}\quad H=v\ind_{A_1}+w\ind_{A_2}.$$Plugging this function into inequality \eqref{pp}, and using multiplicativity of $\p$, we obtain
\begin{equation*}
\begin{split}
\p^{-1}&\Biggl\{\int_X\p\Biggl(\int_YF(x,y)\,\nu(\da y)\Biggr)\mu(\da x)\Biggr\}=\p^{-1}\Biggl\{\int_X\p\bigl(cG(x)+cH(x)\bigr)\,\mu(\da x)\Biggr\}\\
&=c\p^{-1}\Biggl\{\int_X\p\bigl(G(x)+H(x)\bigr)\,\mu(\da x)\Biggr\}=c\p^{-1}\Biggl(\int_{A_1}+\int_{A_2}\Biggr)\\
&=c\p^{-1}\bigl(b\p(t+v)+b\p(u+w)\bigr)\leq\int_Y\p^{-1}\Biggl(\int_X\p\circ F(x,y)\,\mu(\da x)\Biggr)\nu(\da y)\\
&=\int_{B_1}\p^{-1}\Biggl(\int_X\p\bigl(G(x)\bigr)\,\mu(\da x)\Biggr)\nu(\da y)+\int_{B_2}\p^{-1}\Biggl(\int_X\p\bigl(H(x)\bigr)\,\mu(\da x)\Biggr)\nu(\da y)\\
&=c\p^{-1}\Biggl(\int_X\p\bigl(G(x)\bigr)\,\mu(\da x)\Biggr)+c\p^{-1}\Biggl(\int_X\p\bigl(H(x)\bigr)\,\mu(\da x)\Biggr)\\
&=c\p^{-1}\bigl(b\p(t)+b\p(u)\bigr)+c\p^{-1}\bigl(b\p(v)+b\p(w)\bigr).
\end{split}
\end{equation*}
Therefore, dividing by $\p^{-1}(b)$ and $c$, and defining a~function $\pp_\p\colon\R_+^2\to\R$ by the formula $$\pp_\p(\boldsymbol{t})=\p^{-1}\bigl(\p(t_1)+\p_2(t_2)\bigr)\quad\mbox{for }\boldsymbol{t}=(t_1,t_2)\in\R_+^2,$$we may write $$\pp_\p(t+v,u+w)\leq\pp_\p(t,u)+\pp_\p(v,w)\quad\mbox{for }t,u,v,w\in\R_+$$(the Mulholland inequality; see \cite{mulholland} and \cite[\S 8.8]{kuczma}). By appealing to Theorem \ref{ms} below, due to Matkowski and \'Swi\k{a}tkowski, we infer that $\p$ is a~convex homeomorphism, whence $\p(t)=t^p$ for some $p\geq 1$, and all $t\in\R_+$. Consequently, $\psi(t)=t^{1/p}$ for all $t\in\R_+$ and the proof is completed.
\end{proof}
\begin{theorem}[{\cite[Theorem 2]{matkowski_swiatkowski}}]\label{ms}
If $\p\colon\R_+\to\R_+$ is a bijection and the function $\pp_\p$, defined as above, is subadditive on $\R_+^2$, then $\p$ is a~convex homeomorphism of $\R_+$.
\end{theorem}

We may invert inequality \eqref{pp} and, with obvious changes in the proof of Theorem \ref{M1}, derive the multiplicativity of $\p$ and $\psi$, and the relation $\psi=\p^{-1}$. Next, we may conclude that the map $\pp_\p$ is superadditive, hence, by an analogue of Theorem \ref{ms} (\cite[Theorem 3]{matkowski_swiatkowski}), the function $\p$ is a~concave homeomorphism. Therefore, we obtain the following counterpart of Theorem \ref{M1}:
\begin{theorem}\label{M1_invert}
Let $(X,\Sigma,\mu)$, $(Y,T,\nu)$, $\p$ and $\psi$ be as in Theorem \ref{M1}, but instead of \eqref{pp} assume the reversed inequality. Then, and only then, we have either:
\begin{itemize*}
\item[{\rm (i)}] $\psi|_{\p(\R_+)}=0$, or
\item[{\rm (ii)}$^\prime$] $\p(1)\not=0\not=\psi(1)$ and there exists $p\in (0,1)$ such that $\p(t)=\p(1)t^p$ and $\psi(t)=\psi(1)=t^{1/p}$ for $t\in\R_+$.
\end{itemize*}
\end{theorem}

We may now proceed to our main result.
\begin{theorem}\label{M2}
Suppose that $\mu(\Sigma)=[0,1]$ and $\p,\psi\colon\R_+\to\R_+$ are bijections satisfying $\p(0)=\psi(0)=0$, inequality \eqref{A} and condition {\rm (}$\ast${\rm )}. Then there exist numbers $p,q>1$ with $p^{-1}+q^{-1}=1$ such that $\p(t)=\p(1)t^p$ and $\psi(t)=\psi(1)t^q$ for $t\in\R_+$.
\end{theorem}
\begin{proof}
Take any measure space $(Y,T,\nu)$ such that $[0,\alpha]\subset\nu(T)$ for some $\alpha>1$. We wish to apply Theorem \ref{M1} to the given measure space $(X,\Sigma,\mu)$ and to $(Y,T,\nu)$. To this end, we shall show that
\begin{equation}\label{M21}
\p^{-1}\Biggl\{\int_X\p\Biggl(\int_YF(x,y)\,\nu(\da y)\Biggr)\mu(\da x)\Biggr\}\leq\int_Y\p^{-1}\Biggl(\int_X\p\circ F(x,y)\,\mu(\da x)\Biggr)\nu(\da y)
\end{equation}
for every $F\in\Sa_+(X\times Y)$. So, fix any such function $F$ and define $G\in\Sa_+(X)$ by $$G(x)=\int_YF(x,y)\,\nu(\da y)\quad\mbox{for }x\in X.$$Let $\chi\colon\R_+\to\R_+$ be a~function with $\chi(0)=0$ and such that inequality \eqref{A} becomes equality for $(f,g)=(G,\chi\circ G)$. Define also $\Phi\colon\R_+\to\R_+$ by $\Phi(t)=t\chi(t)$ for $t\in\R_+$. Of course, we may assume that $G\not=0$, since otherwise \eqref{M21} is trivial. For simplicity, we will treat the ratio $(\Phi\circ G)(x)/G(x)$ as zero for all these $x\in X$ for which $G(x)=0$. Noticing that $(\Phi\circ G)/G=\chi\circ G$ and using inequality \eqref{A}, we get:
\begin{equation*}
\begin{split}
\int_X&\Phi\Biggl(\int_YF(x,y)\,\nu(\da y)\Biggr)\mu(\da x)=\int_X\frac{(\Phi\circ G)(x)}{G(x)}\cdot G(x)\,\mu(\da x)\\
&=\int_X\frac{(\Phi\circ G)(x)}{G(x)}\int_YF(x,y)\,\nu(\da y)\mu(\da x)=\int_Y\int_X\frac{(\Phi\circ G)(x)}{G(x)}\cdot F(x,y)\,\mu(\da x)\nu(\da y)\\
&\leq\int_Y\Pa_\p\bigl(F(\cdot,y)\bigr)\Pa_\psi\Bigl(\frac{\Phi\circ G}{G}\Bigr)\,\nu(\da y)=\Pa_\psi\Bigl(\frac{\Phi\circ G}{G}\Bigr)\int_Y\Pa_\p\bigl(F(\cdot,y)\bigr)\,\nu(\da y)\\
&=\Pa_\psi(\chi\circ G)\int_Y\p^{-1}\Biggl(\int_X\p\circ F(x,y)\,\mu(\da x)\Biggr)\nu(\da y).
\end{split}
\end{equation*}
We may divide both sides by $\Pa_\psi(\chi\circ G)$ as $\chi\circ G\not=0$. By doing so, we obtain nothing else but inequality \eqref{M21} because by the choice of $\chi$, we have
$$\int_X(\Phi\circ G)(x)\,\mu(\da x)=\int_XG(x)\cdot(\chi\circ G)(x)\,\mu(\da x)=\Pa_\p(G)\Pa_\psi(\chi\circ G).$$

By virtue of Theorem \ref{M1}, there is $p\geq 1$ such that $\p(t)=\p(1)t^p$ for $t\in\R_+$. By symmetry, there is also $q\geq 1$ such that $\psi(t)=\psi(1)t^q$ for $t\in\R_+$. What is left to be proved is the equality $p^{-1}+q^{-1}=1$.

According to Theorem \ref{matkowski2}, the pair of functions $\p(t)=\p(1)t^p$ and $\psi(t)=\psi(1)t^q$ satisfies inequality \eqref{A} if and only if the function $F\colon\R_+^2\to\R_+$, defined by $F(s,t)=s^{1/p}t^{1/q}$ is concave. This is in turn equivalent to the second G\^ateaux differential $\da^2F(\boldsymbol{a})(\boldsymbol{v},\boldsymbol{v})$ being non-positive for every $\boldsymbol{a}\in (0,\infty)^2$ and $\boldsymbol{v}\in\R^2$. An easy calculation shows that for every $\boldsymbol{a}\in (0,\infty)^2$ we have
\renewcommand{\arraystretch}{2.1}
$$\da^2F(\boldsymbol{a})=\left(\begin{array}{cc}\displaystyle{\frac{1}{p}\Bigl(\frac{1}{p}-1\Bigr)} & \displaystyle{\frac{1}{pq}}\\ \displaystyle{\frac{1}{pq}} & \displaystyle{\frac{1}{q}\Bigl(\frac{1}{q}-1\Bigr)}\end{array}\right)\! ,$$

\vspace{1mm}\noindent
whence the concavity of $F$ is equivalent to $$0\leq\mathrm{det}(\da^2F(a))=\frac{1}{pq}\Bigl(1-\Bigl(\frac{1}{p}+\frac{1}{q}\Bigr)\Bigr)\,\,\Longleftrightarrow\,\,\frac{1}{p}+\frac{1}{q}\leq 1.$$

Now, suppose we have the strict inequality $p^{-1}+q^{-1}<1$ and pick any $p^\prime<p$ such that ${p^\prime}^{-1}+q^{-1}<1$. Let $\w\p(t)=\p(1)t^{p^\prime}$. Then for any non-constant function $f\in\Sa_+$ we would have $\Pa_{\w\p}(f)<\Pa_\p(f)$ (see, {\it e.g.}, \cite[\S 3.11]{hardy}), so inequality \eqref{A} holds true after replacing $\p$ by $\w{\p}$. However, this would imply that the original inequality is strict for any non-constant map $f\in\Sa_+$ and any non-zero map $g\in\Sa_+$, which contradicts condition ($\ast$).
\end{proof}

Now, we wish to derive a~counterpart of Theorem \ref{M2} for reversed H\"older's inequality (see, {\it e.g.}, \cite[Theorem 13.6]{hewitt}):
\begin{thm}
Let $0<p<1$ and $q$ satisfy $p^{-1}+q^{-1}=1$ {\rm (}note that $q<0${\rm )} and $\p\colon\R_+\to\R_+$, $\psi\colon (0,\infty)\to (0,\infty)$ be given as $\p(t)=t^p$ and $\psi(t)=t^q$. Then
\begin{equation}\label{R_Holder}
\int_Xfg\,\da\mu\geq\Pa_\p(f)\Pa_\psi(g)
\end{equation}
for all non-negative functions $f\in L_p(\mu)$ and $g\in L_q(\mu)$, unless $\int_Xg^q\,\da\mu=0$.
\end{thm}
\noindent
Note that inequality \eqref{R_Holder} is claimed only for $\mu$-almost everywhere positive $g\in L_q(\mu)$, and that the function $\psi$ is defined only on $(0,\infty)$. When considering step functions $f,g$ these restrictions may be disregarded, provided we define $\psi(0)=0$ and replace $f$ by $f\cdot\ind_{\mathrm{supp}(g)}$ (where $\mathrm{supp}(g)=\{x\in X\colon g(x)\not=0\}$) in inequality \eqref{R_Holder}. Anyway, aiming for a~converse theorem to the reversed H\"older inequality, we shall slightly modify our assumptions in comparison to these of Theorem \ref{M2}. Note also that inequality \eqref{R_Holder} becomes equality if and only if the functions $g^{-1}$ and $f^pg^{-q}$ are proportional, so condition ($\ast$), after adapting to this new situation, again seems natural:
$$\begin{array}{l}
\mbox{For every non-zero function }f\in\Sa_+\mbox{ there exists a~function }\chi\colon\R_+\to\R_+\\
\mbox{satisfying }\chi(0)=0\mbox{ and }\chi(t)>0\mbox{ for }t>0\mbox{ and such that inequality \eqref{R_Holder}}\\
\mbox{becomes equality for }g=\chi\circ f.\mbox{ Conversely, for every non-zero }g\in\Sa_+\\
\mbox{there is a function }\tau\colon\R_+\to\R_+\mbox{ satisfying }\tau(0)=0\mbox{ and }\tau(t)>0\mbox{ for}\\
t>0\mbox{ and such that inequality \eqref{R_Holder} becomes equality for }f=\tau\circ g.
\end{array}
\leqno(**)
$$
\begin{theorem}\label{M3}
Let $\mu(\Sigma)=[0,1]$ and $\p,\psi\colon\R_+\to\R_+$ be bijections with $\p(0)=\psi(0)=0$. Suppose that
\begin{equation}\label{RH}
\int_Xfg\,\da\mu\geq\Pa_\p(f\cdot\ind_{\mathrm{supp}(g)})\Pa_\psi(g)\quad\mbox{for }f,g\in\Sa_+
\end{equation}
and the condition {\rm (}$\ast\ast${\rm )} holds true. Then there exist numbers $p\in (0,1)$ and $q<0$ with $p^{-1}+q^{-1}=1$ such that $\p(t)=\p(1)t^p$ and $\psi(t)=\psi(1)t^q$ for $t\in\R_+$.
\end{theorem}
\begin{proof}
Take any measure space $(Y,T,\nu)$ such that $[0,\alpha]\subset\nu(T)$ for some $\alpha>1$. Observe that we may safely re-write the calculations from the proof of Theorem \ref{M2} using our assumption \eqref{RH} instead of \eqref{A}. By doing so, we apply inequality \eqref{RH} only for the pairs $(f,g)=\bigl(F(\cdot,y),\chi\circ G\bigr)$, where $y\in Y$. This is legitimate as for every such pair we have $f=f\cdot\ind_{\mathrm{supp}(g)}$ because $(\chi\circ G)(x)=0$ implies  $G(x)=0$, thus $F(x,y)=0$. Consequently, we obtain the inequality reverse to \eqref{M21}, for every $F\in\Sa_+(X\times Y)$. In view of Theorem \ref{M1_invert}, we conclude that for some $p\in (0,1)$ we have $\p(t)=\p(1)t^p$ ($t\in\R_+$).

Of course, we cannot repeat the argument above for the function $\psi$ (which is not supposed to be a~homeomorphism of $\R_+$), so let us proceed another way. First, we prove that $\psi$ is continuous on $(0,\infty)$. This will be done with the aid of a~`convex' version of Theorem \ref{matkowski2} (see \cite[Remark 6]{matkowski (studia)}), whose proof we repeat below for completeness.

Take any $\lambda\in(0,1)$ and pick any $A\in\Sigma$ with $\mu(A)=\lambda$ (then $\mu(X\setminus A)=1-\lambda$). For arbitrary $t,u,v,w>0$ let $$f=\p^{-1}(t)\ind_A+\p^{-1}(u)\ind_{X\setminus A}\quad\mbox{and}\quad g=\p^{-1}(v)\ind_A+\p^{-1}(w)\ind_{X\setminus A}.$$Putting these two functions into \eqref{RH} gives $$\lambda F(t,v)+(1-\lambda)F(u,w)\geq F\bigl(\lambda(t,v)+(1-\lambda)(u,w)\bigr),$$where $F\colon (0,\infty)^2\to (0,\infty)$ is defined by $F(s,t)=\p^{-1}(s)\psi^{-1}(t)$. Hence, $F$ is convex on $(0,\infty)^2$, thus it is also continuous (see, \cite[Theorem 7.1.1]{kuczma}). But we already know that $\p$ is continuous. Consequently, $\psi$ is continuous (on $(0,\infty)$) as well.

Now, let $q<0$ be the number satisfying $p^{-1}+q^{-1}=1$ and denote $\gamma(t)=t^q$ for $t>0$. Let $g$ be an arbitrary positive step function and let $\tau$ be a~function from condition ($\ast\ast$). Then, by H\"older's inequality \eqref{R_Holder}, we have $$\Pa_\psi(g)=\frac{\displaystyle{\int_X(\tau\circ g)\cdot g\,\da\mu}}{\displaystyle{\Pa_\p(\tau\circ g)}}\geq\Pa_\gamma(g).$$On the other hand, after substituting $f=g^{(-1+q)/p}$ we get equality in \eqref{R_Holder}, that is, $$\Pa_\p\bigl(g^{(-1+q)/p}\bigr)\Pa_\psi(g)\leq\int_Xg^{(-1+q)/p}\cdot g\,\da\mu=\Pa_\p\bigl(g^{(-1+q)/p}\bigr)\Pa_\gamma(g)\leq\Pa_\p\bigl(g^{(-1+q)/p}\bigr)\Pa_\psi(g).$$
Consequently, $\Pa_\psi(g)=\Pa_\gamma(g)$ for every positive step function $g$. In particular, putting $g=t\ind_{A}+u\ind_{X\setminus A}$, where $\mu(A)=1/2$ and $t,u>0$, we get $$\psi^{-1}\Bigl(\frac{\psi(t)+\psi(u)}{2}\Bigr)=\Bigl(\frac{a^q+b^q}{2}\Bigr)^{1/q}\quad\mbox{for }t,u>0.$$ Hence, the map $(0,\infty)\ni t\mapsto\psi(t^{1/q})$ satisfies Jensen's functional equation and, being continuous, is of the form $\psi(t^{1/q})=at+b$ (see \cite[\S 13.2]{kuczma}). Since $\psi$ maps $(0,\infty)$ onto itself, we conclude that $b=0$, so $\psi(t)=\psi(1)t^q$ and the proof is completed.
\end{proof}

%%%%%%%%%%%%%%%%%%%%%%%%%%%%%%%%%%%%%%%%%%%
%%%%%%%%%%%%%%%%%%%%%%%%%%%%%%%%%%%%%%%%%%%
\section{Remarks on the assumptions}
Let us explain that the assumptions upon the measure space $(Y,T,\nu)$ in Theorem \ref{M1} are essential, not only the inequality $\nu(Y)>1$, but also the requirement that there is a~non-atomic part having measure greater $1$, plays an important role.

First, consider the case where $Y=\{y_1,y_2\}$, $\nu\{y_1\}=\nu\{y_2\}=1$ and $(X,\Sigma,\mu)$ is an arbitrary measure space with $[0,1]\subset\mu(\Sigma)$. Then, inequality \eqref{pp} reads as $$\psi\Biggl\{\int_X\p\bigl(f(x)+g(x)\bigr)\mu(\da x)\Biggr\}\leq\psi\Biggl(\int_X\p\bigl(f(x)\bigr)\mu(\da x)\Biggr)+\psi\Biggl(\int_X\p\bigl(g(x)\bigr)\mu(\da x)\Biggr),$$ where $f(x)=F(x,y_1)$ and $g(x)=F(x,y_2)$. If $\psi=\p^{-1}$ then this is nothing else but Minkowski's inequality $\Pa_\p(f+g)\leq\Pa_\p(f)+\Pa_\p(g)$ (for all $f,g\in\Sa_+(X)$). However, according to Matkowski's result, \cite[Theorem 3]{matkowski (pams)}, such an inequality is equivalent to the fact that the function $\pp_\p$ is concave. Hence, in this case inequality \eqref{pp} obviously does not imply that $\p$ and $\psi$ are power functions.

Now, consider the case where $Y$ and $X$ are probabilistic spaces and let again $\psi=\p^{-1}$, where $\p\colon\R_+\to\R_+$ is an increasing bijection. Recall that the weighted quasi-arithmetic mean with generator $\p$ is given by $$\mathfrak{M}_\p(\boldsymbol{a})=\mathfrak{M}_\p(\boldsymbol{a},\boldsymbol{q})=\p^{-1}\Biggl\{\sum_{j=1}^nq_j\p(a_j)\Biggr\},$$for any finite sequence $\boldsymbol{a}=(a_1,\ldots ,a_n)$ of non-negative numbers, and any sequence $\boldsymbol{q}=(q_1,\ldots ,q_n)$ of weights, {\it i.e.}, non-negative numbers summing up to $1$. There is a~classical result (see \cite[Theorem 106(i)]{hardy}) which says that whenever 
\begin{itemize*}
\item[(i)] $\p$ is four times continuously differentiable on $(0,\infty)$ and
\item[(ii)] the functions: $\p$, $\p^\prime$ and $\p^{\prime\prime}$ are positive on $(0,\infty)$, 
\end{itemize*}
then the inequality
\begin{equation}\label{M_Hardy}
\mathfrak{M}_\p\Bigl(\frac{\boldsymbol{a}+\boldsymbol{b}}{2}\Bigr)\leq\frac{1}{2}\bigl(\mathfrak{M}_\p(\boldsymbol{a})+\mathfrak{M}_\p(\boldsymbol{b})\bigr)
\end{equation}
holds true for all non-negative sequences $\boldsymbol{a}$, $\boldsymbol{b}$, and all non-negative weights if and only if 
\begin{itemize*}
\item[(iii)] the function $\p^\prime/\p^{\prime\prime}$ is concave. 
\end{itemize*}
Suppose that $\p$ has all the properties (i)-(iii). By routine arguments, we may then conclude that inequality \eqref{M_Hardy} holds true in the integral form, that is, when we replace the arithmetic mean at the both sides of \eqref{M_Hardy} by integrating with respect to some probabilistic measure. Of course, we may do the same thing with the weighted mean in the definition of $\mathfrak{M}_\p$. Hence, inequality \eqref{M_Hardy} takes the form $$\p^{-1}\Biggl\{\int_X\p\Biggl(\int_YF(x,y)\,\nu(\da y)\Biggr)\mu(\da x)\Biggr\}\leq\int_Y\p^{-1}\Biggl(\int_X\p\circ F(x,y)\,\mu(\da x)\Biggr)\nu(\da y)$$(for all $F\in\Sa_+(X\times Y)$), which is nothing else but inequality \eqref{pp} with $\psi=\p^{-1}$. However, conditions (i)-(iii) obviously do not imply that $\p$ is a~power function.

The same remarks, with obvious changes, are valid for Theorem \ref{M1_invert}. 

Concerning the assumption $\mu(\Sigma)=[0,1]$ let us pose the following question: Under what weaker assumptions upon the measure space $(X,\Sigma,\mu)$ the assertions in Theorems \ref{M2} and \ref{M3} remain true?
%%%%%%%%%%%%%%%%%%%%%%%%%%%%%%%%%%%%%%%%%%%
%%%%%%%%%%%%%%%%%%%%%%%%%%%%%%%%%%%%%%%%%%%
\bibliographystyle{amsplain}

\end{document}